\documentclass[10pt, a4paper]{article}
\usepackage[margin=1in]{geometry}
\usepackage{graphicx}
\usepackage[]{amsmath}
\usepackage{amsthm}
\usepackage{amsfonts}
\usepackage[utf8]{inputenc}
\usepackage{xcolor}
\usepackage{amssymb}
\usepackage{hyperref}
\usepackage{fullpage}
\graphicspath{ {./images/} }

\newtheorem{theorem}{Theorem}
\newtheorem{lemma}[theorem]{Lemma}
\newtheorem{definition}[theorem]{Definition}
\theoremstyle{remark}
\newtheorem{remark}[theorem]{Remark}
\newtheorem{claim}{Claim}

\newcommand*{\E}{\ensuremath{\mathbb{E}}}

\title{Randomly perturbed digraphs also have bounded-degree spanning trees}
\author{
Patryk Morawski\thanks{Department of Computer Science, ETH, 8092 Z\"urich, Switzerland.
			Email: \href{mailto:pmorawski@student.ethz.ch}{\nolinkurl{pmorawski@student.ethz.ch}}. This author was supported by a scholarship of the Studienstiftung des Deutschen Volkes.}
\and Kalina Petrova\thanks{Institute of Science and Technology Austria (ISTA), Austria. This research was conducted while the author was at Department of Computer Science, ETH Z\"urich, Switzerland.
			Email: \href{mailto:kalina.petrova@inf.ethz.ch}{\nolinkurl{kalina.petrova@ist.ac.at}}. This author was supported by grant no. CRSII5 173721 of the Swiss National Science Foundation.}}
\date{}

\begin{document}
\maketitle
\begin{abstract}
    We show that a randomly perturbed digraph, where we start with a dense digraph $D_{\alpha}$ and add a small number of random edges to it, will typically contain a fixed orientation of a bounded-degree spanning tree.
    This answers a question posed by Araujo, Balogh, Krueger, Piga and Treglown and generalizes the corresponding result for randomly perturbed graphs by Krivelevich, Kwan and Sudakov.
    More specifically, we prove that there exists a constant $c = c(\alpha, \Delta)$ such that if $T$ is an oriented tree with maximum degree $\Delta$ and $D_\alpha$ is an $n$-vertex digraph  with minimum semidegree $\alpha n$, then the graph obtained by adding $cn$ uniformly random edges to $D_\alpha$ will contain $T$ with high probability.
\end{abstract}

\section[short]{Introduction}
Given two graphs $H$ and $G$ with the same number of vertices $n$, a natural question to ask is whether $G$ contains a copy of $H$.
In general, this question is however computationally hard to answer.
Therefore, we might look for conditions that imply existence of a particular subgraph in a host graph $G$.
For example, a classical result of Dirac \cite{dirac} states that if $G$ has minimum degree at least $n/2$, then it has a Hamilton cycle: a cycle going through all of its vertices.
More recently, similar results have been obtained for different classes of spanning subgraphs.
For instance, Komlós, Sárközy and Szemerédi \cite{komlós_sárközy_szemerédi} proved that for any $\Delta$, $\epsilon > 0$ and $n$ large enough, any graph on $n$ vertices with minimum degree $(1/2 + \epsilon)n$ contains every spanning tree with maximum degree at most $\Delta$. This was further strengthened in two different ways --- the maximum degree condition on the tree was improved to $c n/\log n$~\cite{kss2001}, and the minimum degree condition on the host graph for constant-degree trees was relaxed to $n/2+c\log n$~\cite{csaba2010}, both of which are best possible up to the constant $c$. Similar minimum degree conditions have also been studied for different notions of spanning trees in hypergraphs~\cite{pavez-signe_dirac-type_2021,pehova2023minimum}.

These Dirac-type theorems impose very strong density conditions on the host graph and the constant $1/2$ is in fact tight.
Therefore, we might ask ourselves what the situation is in a ``typical'' graph with $n$ vertices and $m$ edges.
We consider a random graph $R \sim G(n, p)$ on $n$ vertices where each edge is included independently and with probability $p$.
Here, it turns out that we can get away with a much lower density of the edges and still typically get a Hamiltonian graph.
A result by P\'osa \cite{posa_hamiltonian_1976} and independently by Korshunov~\cite{korshunov1976solution} states that if $p \geq C\log n / n$ for some $C>0$, such a random graph is w.h.p.\footnote{With high probability (w.h.p.) means with probability tending to $1$ as $n \rightarrow \infty$} Hamiltonian. The threshold for Hamiltonicity  in the binomial random graph has since then been determined more precisely~\cite{komlos1983limit}, and stronger \emph{hitting-time} results have also been shown~\cite{ajtai1985first,bollobas1984evolution}.
More recently, Montgomery \cite{montgomery2014embedding,montgomery_spanning_2019} proved that for any $\Delta > 0$ and a spanning tree $T$ with maximum degree $\Delta$, there exists a constant $c = c(\Delta)$ such that the random graph $R \sim G(n, c\log n / n)$ contains $T$ with high probability.
The exact behaviour of the constant $c$  is still unknown, except for restricted classes of trees \cite{hefetz_spanning_threshold, montgomery2014sharp}.

In \cite{bohman_perturbed}, Bohman, Frieze and Martin introduced a different random graph model which offers a middle way perspective between the deterministic and probabilistic settings introduced above.
It starts with an arbitrary dense $n$-vertex graph $G_\alpha$ with minimum degree at least $\alpha n$ for some constant $\alpha >0$ and adds relatively few random edges to it.
The graph $G_\alpha$ itself might be far away from having the desired property and the underlying question is how large the edge probability $p$ must be for the \emph{randomly perturbed graph} $G = G_\alpha \cup G(n,p)$ to possess the given property with high probability.
For example, in their paper~\cite{bohman_perturbed}, Bohman, Frieze and Martin showed that there exists a constant $c = c(\alpha)$ such that the resulting graph $G = G_\alpha \cup G(n, c/n)$ is Hamiltonian with high probability.
Similarly, Krivelevich, Kwan and Sudakov \cite{krivelevich_bounded-degree_2016} proved that for each $0 < \alpha < 1/2$ and $\Delta > 0$ there exists a constant $c = c(\alpha, \Delta)$ such that if $T$ is a tree with maximum degree $\Delta$ then w.h.p. $G_\alpha \cup G(n, c/n)$ contains $T$.
This result was further improved by Böttcher, Han, Kohayakawa, Montgomery, Parczyk, and Person~\cite{undirected_spanning_universality} who showed that there exists $c = c(\alpha, \Delta)$ such that w.h.p. $G_\alpha \cup G(n, c/n)$ contains all spanning trees with maximum degree at most $\Delta$ simultaneously. Another extension of the result by Krivelevich, Kwan and Sudakov was achieved by Joos and Kim~\cite{joos2020spanning}, who pinned down the order of magnitude of the necessary number of random edges required for spanning trees of maximum degree up to $cn/\log{n}$. Some interesting results on powers of Hamilton cycles and general bounded-degree graphs in randomly perturbed graphs were proved in~\cite{bottcher2020embedding}.

Similar questions were asked for \emph{directed graphs}, or \emph{digraphs}, where the minimum degree condition is replaced with a minimum semidegree one.
For a digraph $D = (V, E)$, its \emph{minimum semidegree} is the minimum of all in- and out-degrees of vertices of $D$, i.e.,
\[ \delta^0(D) := \min_{v\in V(D)} \min\{ \delta^+(v), \delta^-(v) \},\]
where $\delta^+(v)$ is the out-degree of $v$ and $\delta^-(v)$ is the in-degree of $v$.
We call a digraph $D$ an \emph{oriented graph} if there is at most one edge between any pair of vertices or equivalently if it can be obtained from an undirected graph by giving an orientation to each of its edges.
As a consequence of a result by Ghouila-Houri~\cite{ghouilahouri1960condition}, we know that every digraph with minimum semidegree at least $n/2$ contains a consistently oriented Hamilton cycle, and more recently, DeBiasio, Kühn, Molla, Osthus and Taylor \cite{debiasio_all_orientations} showed that a minimum semidegree of $n/2 + 1$\footnotemark enforces the existence of every orientation of the Hamilton cycle in any digraph.
In 2019, Mycroft and Naia \cite{mycroft_spanning_2019} showed that, similarly to the undirected case, for any $\Delta$, $\epsilon$ and for any $n$ large enough, every $n$-vertex digraph with minimum semidegree at least $(1/2 + \epsilon) n$ contains every oriented tree with maximum total degree at most $\Delta$, where the total degree of a vertex is the sum of its in-degree and out-degree. This result was strengthened in a manner again analogous to the undirected setting by Kathapurkar and Montgomery~\cite{kathapurkar22}, who proved that the maximum total degree condition on the tree can be relaxed to $cn/\log{n}$ for some $c>0$. A corresponding counting result was recently shown~\cite{joos2023counting}.

\footnotetext{
In fact, they showed that a digraph with minimum semidegree at least $n/2$ must contain every orientation of a Hamilton cycle except for the anti-directed Hamilton cycle, where the edges are oriented clockwise and anti-clockwise alternately.
The existence of an anti-directed Hamilton cycle can be guaranteed only by semidegree $n/2 + 1$.
}

In \cite{montgomery_spanning_2021}, Montgomery established the threshold for appearance of any fixed orientation of a Hamilton cycle in the binomial random digraph model $D(n, p)$, where each of the possible $2\binom{n}{2}$ edges is present independently and with probability $p$.
He showed that depending on the orientation the threshold may vary from $p = \log n /n$ to $p = 2 \log n /n$.
Moreover, by his observation in the same paper, which is a result of McDiarmid's coupling argument \cite{mcdiarmid_general_1983}, we know that for each $\Delta > 0$ and an oriented tree $T$ with maximum total degree $\Delta$ there exists a constant $c = c(\Delta)$ such that w.h.p. $D(n, c \log n/ n)$ contains $T$.

For the \emph{randomly perturbed digraph} model, where we start with a graph $D_\alpha$ with minimum semidegree $\alpha n$ and add random edges to it, Bohman, Frieze and Martin \cite{bohman_perturbed} showed that for every $\alpha > 0$ there exists $c = c(\alpha)$ such that w.h.p. $D_\alpha \cup D(n, c/n)$ contains a consistently oriented Hamilton cycle.
Later, Krivelevich, Kwan, and Sudakov~\cite{krivelevich2016cycles} showed that w.h.p. in $D_\alpha \cup D(n, c/n)$ one can find every consistently oriented cycle of length $2$ to $n$ and more recently, Araujo, Balogh, Krueger, Piga, and Treglown~\cite{araujo_oriented_2022} proved that the same model contains every orientation of a cycle of length $2$ to $n$, in both cases for some $c = c(\alpha)$.
In the same work, the latter set of authors also asked the question whether a similar statement holds for a fixed bounded-degree spanning tree in randomly perturbed digraphs.

In this paper, we answer this question in the positive by giving a short proof of the following theorem.
\begin{theorem}\label{theorem:maintheorem}
    For each $0 < \alpha <1$ and $ \Delta \in \mathbb{N}$, there exists a constant $c = c(\alpha, \Delta)$ such that if $D_\alpha$ is an $n$-vertex digraph with minimum semidegree at least $\alpha n$,
    and $R \sim D(n,c/n)$, and $T$ is an $n$-vertex oriented tree with maximum total degree at most $\Delta$, then w.h.p.\footnotemark $T \subseteq D_\alpha \cup R$.
\end{theorem}
\footnotetext{More formally, the theorem states that if we fix any pair of sequences $(T^{(n)})_{n \geq 1}$ and $(D_\alpha^{(n)})_{n \geq 1}$ such that $T^{(n)}$ is an $n$-vertex oriented tree with maximum total degree at most $\Delta$ and $D_\alpha^{(n)}$ is an $n$-vertex digraph with minimum semidegree at least $\alpha n$, then $\lim_{n \to \infty} \Pr[T^{(n)} \subseteq D_\alpha^{(n)} \cup D(n, c/n)] = 1$.}

The number of random edges added to $D_\alpha$ is tight up to the dependence of $c$ on $\Delta$ and $\alpha$, as the example of a digraph obtained by replacing every edge in the complete bipartite graph $K_{\alpha n, (1-\alpha)n}$ by two directed edges in both directions shows. The same extremal example also demonstrates why the minimum semidegree condition is best possible. Namely, if we allow $\alpha = o(1)$ and consider the larger part of the bipartite graph in question, there will be linearly many vertices of degree $0$ in $D(n,c/n)$. All of these would then need to be attached to the spanning tree via an edge to the other part of $K_{\alpha n, (1-\alpha)n}$, which does not have enough vertices for that --- it only has $o(n)$ many.

An open question remains whether one could allow $\Delta$ to modestly grow with $n$, just like is the case for minimum degree conditions in dense graphs and digraphs~\cite{kathapurkar22,kss2001}.
That would require significantly new ideas, since the embedding of most of the tree would no longer fit in the random edges at density $\frac{c}{n}$. Thus, one would have to also make use of the deterministic edges for that, which would interfere with the probabilistic argument for finding absorbing stars. 

\subsection*{Proof outline}

To prove Theorem \ref{theorem:maintheorem}, we will show that with high probability we can find an embedding of $T$ in $D$ in two steps.
First, we will fix a subtree $T'$ of $T$ with $(1-\epsilon)n$ vertices for some small but constant $\epsilon > 0$ and we will show that w.h.p. $T' \subseteq R$, i.e., that we can find an embedding of $T'$ in $D$ which uses only the random edges.
Then we will show that using the edges of $D_\alpha$ we can w.h.p. extend this embedding to an embedding of the whole tree $T$.
Our proof strategy works for both undirected and directed randomly perturbed graphs, the only technicality in the directed case being taking care of edge directions, which we will omit in this proof outline.

To show that w.h.p. $R$ contains the almost spanning tree $T'$, we will rely on the corresponding result for undirected graphs by Alon, Krivelevich and Sudakov \cite[Theorem 1.1]{alon_embedding_2007}. 
They showed that for each $\epsilon \in [0,1]$ and each $\Delta >0$ there exists a constant $c = c(\epsilon, \Delta)$ such that w.h.p. $G(n, c/n)$ contains every tree on $(1-\epsilon)n$ vertices with maximum degree at most $\Delta$.
Using a consequence of the coupling argument by McDiarmid \cite{mcdiarmid_general_1983}, as observed by Montgomery \cite{montgomery_spanning_2021}, we will be able to get that w.h.p.  there exists an embedding $\phi'$ of $T'$ in $ R$ where $R \sim D(n, c/n)$ for the same choice of $c$.
\begin{figure}
    \centering
    \includegraphics[width=0.55\textwidth]{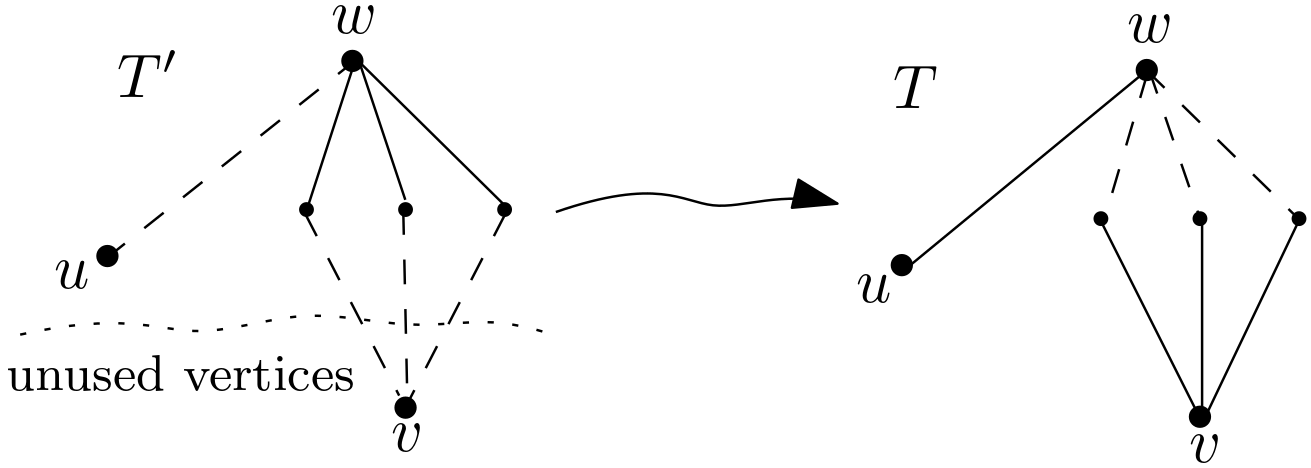}
    \caption{An illustration of the absorbtion strategy used. We wish to extend our tree by an edge with an endpoint in $u$. To achieve that we pick an unused vertex $v$ as well as a vertex $w$ that is part of our tree and such that $v$ shares all the tree-neighbors of $u$. To extend the tree, we let $v$ take over the role of $w$ and $uw$ be the new edge.}
    \label{fig:absobtion}
\end{figure}
To extend $\phi'$ to an embedding $\phi$ of the whole tree in $D$, we will use an absorption strategy inspired by the one developed for undirected graphs in \cite{bottcher2020embedding} and later also employed for hypergraphs in \cite{pavez-signe_dirac-type_2021,pehova2023minimum}.
We will extend the embedding of $T$ edge by edge.
In each step, we will create a new edge with one endpoint in some vertex $u$ that is already part of our embedding.
For it, we will pick some vertex $w$ that is still not used by the embedding and find a vertex $v$ such that $w$ shares all the relevant neighbors of $v$ and there is an edge between $v$ and $u$.
This will allow $w$ to take over the role of $v$ while the edge between $u$ and $v$ will serve as the new edge in our embedding.
An illustration of this process can be found in Figure \ref{fig:absobtion}.
We call such a structure an \emph{absorbing star} with center $v$ for the pair $u$, $w$.

We will then show that w.h.p.\hspace{0.35em}enough such absorbing stars exist for each pair under $\phi'$, which will guarantee that in each step we will be able to find such a vertex $v$. For this, inspired by the techniques used for the undirected case in~\cite{bottcher2020embedding,undirected_spanning_universality}, we will employ the randomness inherent to our host graph in the following way.
We will first fix a collection of stars in our tree and then show that w.h.p. they are nicely distributed in $D_\alpha \subseteq D$.
Here, we will make use of the observation by \cite{krivelevich_bounded-degree_2016} that by symmetry, since $\phi'$ only uses the random edges of $R$, we can assume it is a uniformly random injective mapping of the vertices of $T'$ into the vertices of $D_{\alpha}$.
Under this assumption, we will first show that the expected number of absorbing stars for each pair of vertices is large.
Then, using the Azuma-Hoeffding inequality on a martingale defined by exposing the positions of the structures in the host graph one by one, we will show that for each pair this number is highly concetrated around its expectation and thus, w.h.p. we will be able to find enough such structures.

\vspace{10pt}

The proof of Theorem \ref{theorem:maintheorem} is organized as follows.
In Section \ref{section:almostembedding} we will show that w.h.p. we can embed all but a bounded small proportion of $T$ in $R$, relying on the result in~\cite{alon_embedding_2007} for undirected trees in binomial random graphs.
In Section \ref{section:absorption} we will present the absorption strategy we use to complete an embedding of an oriented tree.
Then, in Section \ref{section:manystars} we will show that w.h.p. the conditions needed for the method to work are satisfied after embedding almost all of $T$ in $R$ and thus we will be able to extend our embedding to the whole $T$.
Finally, in Section \ref{section:proofofmain} we put all our results together to prove Theorem \ref{theorem:maintheorem}.

\section{Preliminaries}
\subsection*{Notation}
Throughout this paper we omit floors and ceilings whenever it does not affect the argument.
We write $n^{\underline{i}}$ for $\frac{n!}{(n-i)!}$.
Given a digraph $D$, we write $V(D)$ for the set of vertices and $E(D)$ for the set of (directed) edges of $D$.
For $v \in V(D)$, we denote by $N_D^+(v) = \{u | (v, u) \in E(D)\}$ the set of out-neighbors of $v$ and by $N_D^-(v) = \{u | (u, v) \in E(D)\}$ the set of in-neighbours of $v$.
The underlying graph $G$ of $D$ is the undirected (multi-)graph obtained by replacing each directed edge $(v, w)$ in $D$ with an undirected edge $\{v, w\}$.
A directed graph $D$ is an (oriented) tree iff its underlying graph is a tree.

We call an injective mapping $\phi: V(H) \to V(D)$ an \emph{embedding} of $H$ in $D$ if for any $(v, w) \in E(H)$, we have that $(\phi(v), \phi(w)) \in E(D)$.
We denote by $Im(\phi) = \{ \phi(x) | x \in V(H)\}$ the image of $H$ and write $\phi(S) = \{\phi(x) | x \in S \}$ for $S \subseteq V(H)$.
Note that $H \subseteq D$ iff there exists an embedding of $H$ into $D$.

Let $T$ be an oriented tree.
For $S \subseteq E(T)$ we write $V(S) = \{u,v | (u,v) \in S \}$ for all the vertices that are incident to some edge in $S$.
We call an ordering $e_1,\dots, e_{n-1}$ of edges of $T$ a \emph{valid ordering} if, for all $i \in [n-2]$, we have $|V(\{e_1, \dots, e_i\}) \cap V(\{e_{i+1}\})| = 1$.
It is easy to see that every oriented tree has a valid ordering of its edges.

\subsection*{Martingales and the Azuma-Hoeffding inequality}
In the proof of Lemma \ref{lemma:enoughstars}, to prove that the probability for certain bad events is sufficiently small we will use the Azuma-Hoeffding inequality for martingales. 
Intuitively, a (discrete-time) martingale is a sequence of random variables $(M_t)_{t \geq 0} \in \mathbb{R}$ where $M_t$ is a state of a given process at time $t$.
Moreover, knowing the state $M_t$ of the process at time $t \geq 0$, the expected value of the next state must be equal to $M_t$, i.e., $\E[M_{t+1} | M_1, \dots, M_t] = M_t$.

For our purposes, it will be enough to look at the process of exposing more and more information that a given random variable $X$ depends on.
More formally, let $(\Omega, \mathcal{F}, \Pr)$ be a probability space and $X$ be a random variable on that probability space.
Moreover, let $(\mathcal{F}_t)_{t \geq 0} $ be a sequence of $\sigma$-algebras such that for all $i \geq 0$ it holds that $\mathcal{F}_i \subseteq \mathcal{F}_{i+1} \subseteq \mathcal{F}$, i.e., each subsequent $\mathcal{F}_i$ contains more information than the previous one.
Then we call $(\mathcal{F}_t)_{t \geq 0}$ a \emph{filtration} and the sequence $(\E[X | \mathcal{F}_t])_{t \geq 0}$ is a \emph{martingale}.

We will specifically look at martingales $(M_t)_{t \geq 1}$ such that the change in each step, i.e., $|M_{t+1} - M_{t}|$ for all $t\geq 1$, can be almost surely\footnote{Almost surely means with probability 1.} bounded by some constant $L$.
In this case, the Azuma-Hoeffding inequality will allow us to bound the probability that after taking $N$ steps our process changes by a lot, i.e., that $|M_0 - M_N|$ is large.

\vspace{10px}
\begin{theorem}[Azuma-Hoeffding inequality]\cite{azuma, hoeffding}\label{thm:azuma}
        Suppose $(M_t)_{t\geq 0}$ is a martingale and for all $t\geq 1$ it holds that $|M_t - M_{t-1}| \leq L$ almost surely.
        Then for all positive integers $N$ and any $\epsilon > 0$,
        \[ \Pr[|M_N - M_0| \geq \epsilon] \leq 2\exp \Bigg(\frac{-\epsilon^2}{2NL^2} \Bigg). \]
\end{theorem}

In particular, in case $M_t = \E[X | \mathcal{F}_t]$, $\mathcal{F}_0 = \{ \emptyset, \Omega \}$ and $\sigma(X) \subseteq \mathcal{F}_N$ we have that $M_0 = \E[X]$ and $M_N = X$, so the inequality will allow us to bound the probability that $X$ deviates a lot from its expectation.
In the following, for a sequence $(\phi_t)_{t\geq 0}$ of random variables such that $(\sigma(\phi_t))_{t\geq 0}$ is a filtration, we will write $M_t = \E[X | \phi_t]$ for the martingale $M_t = \E[X | \sigma(\phi_t)]$.

\section{Finding an almost embedding}\label{section:almostembedding}

In this section, we will find an embedding of a large fraction of a given tree $T$ in $D \cup R$ using only the random edges of $R$.
In the later sections we will then show how to extend this embedding to an embedding of the whole tree. 

In other words, we will prove the following lemma.

\begin{lemma}[Embedding an almost spanning tree]\label{lemma:almostembedding}
    For all $\epsilon >0$ and $\Delta \in \mathbb{N}$, there exists $c = c(\epsilon, \Delta)$ such that the following holds.
    Let $T'$ be an oriented tree of maximum total degree at most $\Delta$ on $(1-\epsilon)n$ vertices.
    Then w.h.p., there exists an embedding $\phi$ of $T'$ into $R \sim D(n, c/n)$.
\end{lemma}

To prove the lemma, we will translate the corresponding result for undirected graphs \cite[Theorem 1.1]{alon_embedding_2007} using the observation by Montgomery \cite[Theorem 3.1]{montgomery_spanning_2021} which is a consequence of McDiarmid's beautiful coupling argument \cite{mcdiarmid_general_1983}.

Let $D^*(n, p)$ denote a digraph where each possible pair of edges $(u,v)$ and $(v,u)$ is included together, independently from other edges, with probability $p$. 
Clearly, the probability that $R^* \sim D^*(n,p)$ contains an oriented tree $T$ is equal to the probability that $G \sim G(n, p)$ contains the underlying graph of $T$.
Moreover, the following observation by Montgomery will give us a tool to translate the result for $R^* \sim D^*(n, p)$ to $R \sim D(n,p)$.
\begin{lemma}\cite{araujo_oriented_2022,mcdiarmid_general_1983}\label{lemma:coupling}
    Let $\mathcal{H}$ be a collection of oriented graphs. Then,
    \[\Pr[\exists H \in \mathcal{H}: H \subseteq D(n, p)] \geq \Pr[\exists H \in \mathcal{H}: H \subseteq D^*(n, p)].\]
\end{lemma}

Putting those results together will yield us a proof for our lemma.

\begin{proof}[Proof of Lemma \ref{lemma:almostembedding}]

As shown in \cite[Theorem 1.1]{alon_embedding_2007}, there exists $c=c(\epsilon, \Delta)$ such that $G(n, c/n)$ w.h.p. contains every (undirected) tree on $(1-\epsilon)n$ vertices and maximum degree at most $\Delta$.
In particular, for the same choice of $c$, we have that w.h.p. $T' \subseteq D^*(n, c/n)$.
Therefore, initializing Lemma \ref{lemma:coupling} with $\mathcal{H} = \{T' \}$ together with our previous observations yields Lemma \ref{lemma:almostembedding}.

\end{proof}

\begin{remark}\label{remark:uniform}
    In the following sections, we will condition on the event that such an embedding $\phi$ in fact exists.
    As observed by \cite[Remark 3]{krivelevich_bounded-degree_2016}, since $\phi$ only uses the random edges of $R = D(n, p)$, we can assume that $\phi$ is a uniformly random injective mapping of $V(T')$ into $V(D_\alpha)$.
    In particular, for any $S \subseteq V(T')$, $Im(S)$ is a uniformly random subset of $V(D_\alpha)$ of size $|S|$.
    Moreover, conditioned on $\phi(S)$ for some $S \subseteq V(T')$, the embedding $\phi|_{V(T')\setminus S}$ is still a uniformly random injective mapping of $V(T')\setminus S$ into $V(D_\alpha) \setminus Im(S)$.
    We will use this observation heavily in Section \ref{section:manystars}, where we will show that certain substructures of $T'$ are with high probability nicely distributed in our graph $D_\alpha$.
\end{remark}

\section{Completing the embedding}\label{section:absorption}
In this section, we will present the absorption technique, which will allow us to extend an embedding of an almost spanning tree to an embedding of a spanning tree.
To achieve it, let us start with some relevant definitions.
\begin{definition}
    A \emph{star} $S_v = (v, S^+, S^-)$ consists of edges $(v, u)$ for $u \in S^+$ and $(u,v)$ for $u \in S^-$.
    We say that $v$ is the \emph{center} of the star $S_v$.
\end{definition}

\begin{definition}
    Let $T$ be a tree, $D$ a directed graph and $\phi: V(T) \to V(D)$ an injective mapping.
    For $u, w \in V(D)$ and $* \in \{+, -\}$, we say that a star  $S_v = (v, S^+, S^-)$ in $T$ is $(u,*,w)$-absorbing under $\phi$ if:
    \begin{itemize}
        \item $\phi(v) \in N_D^*(u)$,
        \item $S^+ = N_T^+(v)$ and $S^- = N_T^-(v)$, \emph{and}
        \item $\phi(S^+) \subseteq N_D^+(w)$ and $\phi(S^-) \subseteq N_D^-(w)$.
    \end{itemize}
    We write $\mathcal{A}_\phi(u,*,w)$ for the set of all $(u,*,w)$-absorbing stars in $T$ under $\phi$.
\end{definition}
Intuitively, the conditions for an absorbing star will allow us to extend $\phi$ to an embedding of a larger tree,
so that $w$ takes over the role of $\phi(v)$ in the embedding while the new edge between $u$ and $\phi(v)$ is created.
This is formalized in the following lemma, which shows how to complete an embedding of almost all of $T$ in $D \cup R$,
given that enough absorbing stars exist for each pair of vertices. 

\begin{lemma}[Absorbing Lemma]\label{lemma:absorbing}
    Let $n\in \mathbb{N}$ and $T$ be an $n$-vertex oriented tree with a valid edge ordering $e_1, \dots, e_{n-1}$, $i \in [n-2]$ and $T_0 = \{ e_1, \dots, e_{n-1 - i}\}$ be a subtree of $T$ on $n-i$ vertices.
    Let $D$ be an $n$-vertex digraph and let $\phi_0$ be an embedding of $T_0$ into $D$.
    Suppose there exists a family $\mathcal{S}$ of vertex-disjoint stars in $T_0$ such that for all  $ u,w \in V(D), * \in \{+, -\}$, it holds that
    \[ |\mathcal{S} \cap \mathcal{A}_{\phi_0}(u,*,w)| \geq 2i.\]
    Then there exists an embedding of $T$ into $D$.
\end{lemma}
\begin{proof}
    Let $V(D) \setminus Im(\phi_0) = \{x_1, \dots, x_i\}$ and for each $j \in [i]$, let $T_{j} = \{ e_1, \dots, e_{n-1 - i + j}\}$.
    We will inductively find embeddings $\phi_j: T_j \to D$ and subsets $\mathcal{S}_j \subseteq \mathcal{S}$ such that:
    \begin{enumerate}
        \item $Im(\phi_j) = Im(\phi_0) \cup \{x_1, \dots, x_j\}$,
        \item $|\mathcal{S}_j| \leq 2j$,
        \item $|(\mathcal{S}\setminus \mathcal{S}_j) \cap \mathcal{A}_{\phi_j}(u,*,w)| \geq 2(i-j)$ for all $u,w \in V(D), * \in \{+, -\}.$
    \end{enumerate}
    The final embedding will then be given by $\phi_i$.

    Clearly, these conditions hold for $\phi_0$.
    Let us then suppose that we already have $\phi_j$ and $\mathcal{S}_j$ for some $j$.
    Let then $e_{n - i + j} = (v_1, v_2)$ be the next edge to embed and suppose w.l.o.g.\footnotemark that $v_1 \in V(T_j)$.
    \footnotetext{In case $v_2 \in V(T_j)$ we would have to find a $(\phi_j(v_2), -, x_{j+1})$ absorbing star. The analysis is symmetric to the other case.}
    Let moreover $S_v \in (\mathcal{S}\setminus \mathcal{S}_j)$ be an $(\phi_j(v_1), +, x_{j+1})$-absorbing star under $\phi_j$,
    which exists by condition 3.
    We define $\phi_{j+1}$ as:
    \[\phi_{j+1}(v') = \begin{cases}
        x_{j+1} & v' = v \\
        \phi_j(v) & v' = v_2 \\
        \phi_j(v') & \text{otherwise}
    \end{cases}\]

    Note that $\phi_j$ is injective and condition 1 holds by construction.
    Moreover, for all $(u,*,w)$ any $S \in (\mathcal{S}\setminus \mathcal{S}_j) \cap \mathcal{A}_{\phi_j}(u,*,w)$ is
    \begin{itemize}
        \item equal to $S_v$,
        \item a star with center $v_1$ (if it exists), \emph{or}
        \item still $(u,*,w)$-absorbing under $\phi_{j+1}$.
    \end{itemize}
    
    Let $S_{v_1}$ be the star in $\mathcal{S}$ with center $v_1$ in case it exists (and let $S_{v_1}$ = $S_v$ otherwise).
    After setting $\mathcal{S}_{j+1} = \mathcal{S}_j \cup \{ S_v, S_{v_1}\}$ conditions 2 and 3 are therefore also satisfied.
\end{proof}

\section{Finding many absorbing stars}\label{section:manystars}
In this section we will show how to find a collection of stars $\mathcal{S}$ satisfying the conditions of Lemma \ref{lemma:absorbing}.
To that end, we fix some $0 < \epsilon < 1/3$ and let $n' = \lfloor (1- \epsilon)n \rfloor$.
Moreover, we let $e_1, \dots, e_{n-1}$ be a valid ordering of the edges of $T$ and $T' = \{ e_1, \dots, e_{n'-1} \}$ be a subtree of $T$ with $n'$ vertices.
We will first find a collection of many disjoint stars in $T'$ and then show that w.h.p. there are enough absorbing stars among them for any pair of vertices.

We start with finding a suitable collection $\mathcal{S}$ of many vertex disjoint stars in $T'$.

\begin{lemma}[Fixing many vertex disjoint stars in $T'$]\label{lemma:findingstars}
    Let $T'$ be an oriented tree on $n'$ vertices and with maximum total degree $\Delta$.
    There exists a subset $V' \subseteq V(T')$ of $|V'| \geq \frac{n'}{\Delta^2+1}$ vertices of $T'$ such that the collection 
    \[\mathcal{S}_0 := \{ (v, N^+_{T'}(v), N^-_{T'}(v)) | v\in V' \}\]
    is vertex-disjoint.
\end{lemma}
\begin{proof}
    Let us write $S_v = (v, N^+_{T'}(v), N^-_{T'}(v))$ for $v \in V(T)$.
    We start with $\mathcal{S}_0 = \emptyset$ and process the vertices $V(T')$ one by one.
    Whenever for the current vertex $v$ the star $S_v$ is vertex-disjoint to all stars in the current set $\mathcal{S}$,
    we add it to $\mathcal{S}_0$.
    Note that if for $v, w \in V(T')$ the stars $S_v$ and $S_w$ are not vertex-disjoint then $v$ and $w$ must have distance at most $2$ in the underlying graph of $T'$.
    Therefore, for any $v \in V(T')$ there are at most $\Delta^2$ vertices $w$ so that the stars $S_v$ and $S_v$ would conflict.
    This means that at the end of the process we have $|\mathcal{S}_0| \geq \frac{n'}{\Delta^2+1}$.
\end{proof}

We will now show that under a uniformly random embedding of $T'$ into $D$ and for any pair of vertices enough of those stars are absorbing.

\begin{lemma}[There are enough absorbing stars for each pair]\label{lemma:enoughstars} 
    Let $n \in \mathbb{N}$, $0 < \alpha, \gamma \leq 1 $ and $\Delta > 0$. 
    There exists a constant $c_1 = c_1(\alpha, \gamma, \Delta)$ such that the following holds.
    Let $T'$ be a tree on $n/2 \leq n' \leq n$ vertices and with maximal total degree at most $\Delta$ and let $V \subseteq V(T')$ be a subset of its vertices of size $\gamma n$  such that $\mathcal{S} = \{ (v, N_{T'}^+(v), N_{T'}^-(v)) | v \in V\}$ is a set of vertex-disjoint stars in $T'$.
    Moreover, let $D_{\alpha}$ be an $n$-vertex digraph with minimum semidegree at least $\alpha n$ and $\phi: V(T') \to V(D_\alpha)$ be a uniformly random injective mapping of $V(T')$ into $V(D_\alpha)$.
    Then, w.h.p. for all $u,w \in V(D_\alpha)$ and $*\in \{+, - \}$ it holds that
    \[ |\mathcal{S} \cap \mathcal{A}_{\phi}(u, *, w)| \geq c_1 n. \]
\end{lemma}
\begin{proof}
    Let $N := \min \{\gamma n, \frac{\alpha n}{6(\Delta+1)}\}$. It will be enough to consider only the first $N$ stars from $\mathcal{S}$, so assume $|\mathcal{S}|=N$.
    Let us fix $u, w \in V(D_\alpha)$ and $* \in \{+, -\}$.
    To make our analysis simpler, we will restrict ourselves to arbitrary subsets of the neighborhoods of $u$ and $w$ such that there are no overlaps between them.
    We pick pairwise disjoint $N^*(u) \subseteq N_{D_\alpha}^*(u)$, $N^+(w) \subseteq N_{D_\alpha}^+(w)$ and $N^-(w) \subseteq N_{D_\alpha}^-(w)$ of size $\alpha' n = \alpha n/3 $ each and say that a star $S = (v, N_{T'}^+(v), N_{T'}^-(v)) \in \mathcal{S}$ is \emph{good} if 
    \begin{itemize}
        \item $\phi(v) \in N^*(u)$,
        \item $\phi(N_{T'}^+(v)) \subseteq N^+(w)$, \emph{and}
        \item $\phi(N_{T'}^-(v)) \subseteq N^-(w).$
    \end{itemize}
    Clearly, if $S$ is good then it is also $(u, *, w)$-absorbing under $\phi$.

    For $S \in \mathcal{S}$, we define $X_S$ as the indicator variable for the event that $S$ is good and let $X = \sum_{S \in \mathcal{S}} X_S$ be the number of good stars in $\mathcal{S}$.

    \vspace{10pt}
    
    \begin{claim}
        There exists a constant $c_2 =  c_2(\alpha, \gamma, \Delta)$ such that $\E[X] \geq  c_2n$.
    \end{claim}
    \begin{proof}
        Let $S = (v, N_T^+(v), N_T^-(v)) \in \mathcal{S}$, and let  $ s^+ = |N_{T'}^+(v)| $, $s^- = |N_{T'}^-(v)|$ and $s = s^+ + s^- + 1 \leq \Delta + 1$.
        Then we can choose a constant $c'_2 = c'_2(\alpha, \Delta)$ such that for all possible choices of $s^+$ and of $s^-$ it holds that
        \[ \Pr[S \text{ is good}] = \frac{\alpha'n (\alpha'n)^{\underline{s^+}}(\alpha'n)^{\underline{s^-}}}{n^{\underline{s}}} \geq c'_2.\]
        By linearity of expectation we then get the claim.
   \end{proof}
    \vspace{10pt}

    What remains to be shown is that $X$ is concentrated around its mean.
    To achieve that, we will expose the positions of the stars $\mathcal{S}$ in our host graph $D_\alpha$ one by one to define a star-exposure martingale.
    Then we will show that exposing an additional star cannot change the expected number of good stars by much.

    Let therefore $\mathcal{S} = \{S_1, \dots, S_N \}$ be an arbitrary enumeration of the stars in $\mathcal{S}$, and denote $V_i = \bigcup_{j=1}^i V(S_j)$ and $X_i = X_{S_i}$.
    Moreover, let $\phi_i = \phi|_{V_i}$ be the restriction of $\phi$ to the first $i$ stars and:
    \begin{itemize}
        \item $M_0 = \E[X]$
        \item $M_i = \E[X | \phi_i]$ for $1 \leq i \leq N.$
    \end{itemize}
    Clearly, $(M_i)_{i \geq 0}$ is a martingale and $M_N = X$.

    \vspace{10pt}
    
    \begin{claim}
        There exists a constant $c_3 = c_3(\alpha, \Delta)$ s.t. $|M_k - M_{k-1}| \leq c_3$ almost surely for every $k \geq 1$.
    \end{claim}
    \begin{proof}
        Since we're dealing with a finite probability space, it suffices to consider $|M_k - M_{k-1}|$ for any possible embedding $\phi$.
        Let us therefore fix an embedding $\phi$ of the stars in $\mathcal{S}$ into $D_{\alpha}$.
        Then,
        \begin{itemize}
            \item $M_{k-1} = \E[X | \phi_{k-1}] = \sum_{i=0}^{k-1}\E[X_i | \phi_{k-1}] + \E[X_k | \phi_{k-1}] + \sum_{i=k+1}^{N}  \E[X_i | \phi_{k-1}]$
            \item $M_{k} = \E[X | \phi_{k}] = \sum_{i=0}^{k-1}\E[X_i | \phi_{k}] + \E[X_k | \phi_{k}] + \sum_{i=k+1}^{N}  \E[X_i | \phi_{k}].$
        \end{itemize}
        Therefore, $|M_{k} - M_{k-1}| \leq 1 + \sum_{i=k+1}^{N}|\E[X_i | \phi_{k-1}] - \E[X_i| \phi_{k}]|$.
        
        Note that since $\phi$ is a uniformly random injective mapping,
        $\E[X_i | \phi_{k-1}]$ only depends on the sizes of the respective parts $N^*(u), N^+(w), N^-(w)$ and $V(D_{\alpha})$ that are not used by $\phi_{k-1}$.
        Let therefore $a^*$, $a^+$,  $a^-$ and $b$ be such that:
        
        \begin{itemize}
            \item $|N^*(u)\setminus Im(\phi_{k-1})| = a^*n$
            \item $|N^+(w)\setminus Im(\phi_{k-1})| = a^+n$
            \item $|N^-(w)\setminus Im(\phi_{k-1})| = a^-n$ and
            \item $|V(D_{\alpha})\setminus Im(\phi_{k-1})| = b n.$
        \end{itemize}
        Note that by our choice of $|\mathcal{S}|$, we have that $a^*, a^+, a^- \in [\alpha/6, \alpha/3]$ and $1/2 \leq b \leq 1$.

        For $i \geq k+1$, let $S_i = (v_i, S^+, S^-)$, and let $ s^+ = |S^+|$, $ s^-=|S^-|$ and $s = s^+ + s^- + 1 \leq \Delta + 1$. 
        Then, it holds that:
        \[\E[X_i | \phi_{k-1}] = \frac{a^*n (a^+n)^{\underline{s^+}}(a^-n)^{\underline{s^-}}}{(b n)^{\underline{s}}} \]
        and since $|V(S_k)| \leq \Delta + 1$ we also have
        \[\frac{(a^*n - \Delta - 1)(a^+n- \Delta - 1)^{\underline{s^+}}(a^-n- \Delta - 1)^{\underline{s^-}}}{(b n)^{\underline{s}}} \leq \E[X_i | \phi_{k}] \leq \frac{a^*n (a^+n)^{\underline{s^+}}(a^-n)^{\underline{s^-}}}{(b n - \Delta - 1)^{\underline{s}}}.\]
        Thus, there exists a constant $c_4 = c_4(\alpha, \Delta)$ such that for all possible choices of $s$'s, $a$'s and $b$,
        \[|\E[X_i | \phi_{k-1}] - \E[X_i| \phi_{k}]| \leq c_4n^{-1}. \]
        Since $N \leq n$, we get $|M_{k} - M_{k-1}| \leq  1 + c_4$.
    \end{proof}
    \vspace{10pt}

    By the Azuma-Hoeffding inequality (Theorem~\ref{thm:azuma}) we consequently get that there exists a constant $c_5 = c_5(\alpha, \Delta)$ such that:
    \[ \Pr[|X - \E[X]| \geq \E[X]/2] = \Pr[|M_0 - M_N| \geq \E[X]/2] \leq 2 \exp(\frac{-\E[X]^2}{8N c_3^2}) \leq  \exp(-c_5n).\]
    Taking a union bound over all possible choices of $u, w$ and $*$ yields the lemma with $c_1 = c_2/2$. 
\end{proof}

\section{Proof of the main result}
\label{section:proofofmain}

\begin{proof}[Proof of Theorem \ref{theorem:maintheorem}]
Let $\gamma := \frac{1}{2(\Delta^2 + 1)}$, let $c_1 = c_1(\alpha, \gamma, \Delta)$ be the constant from Lemma \ref{lemma:enoughstars}, and $\epsilon < c_1 /2$. 
Let $D_\alpha$ be a digraph on $n$ vertices with minimum semidegree at least $\alpha n$ and $T$ be an oriented tree on $n$ vertices and maximum total degree at most $\Delta$.
We let $e_1, \dots, e_{n-1}$ be a valid ordering of the edges of $T$ and $T' = \{ e_1, \dots, e_{(1 -\epsilon) n-1}\}$ be a subtree of $T$ on $(1 -\epsilon) n$ vertices. 

Now, by Lemma \ref{lemma:almostembedding} there is a constant $c = c(\epsilon, \Delta)$ such that the random digraph $R \sim D(n, c/n)$, and thus also $D = D_\alpha \cup R$, w.h.p. contains $T'$.
Let $\phi' : V(T') \to V(D)$ be an embedding of $T'$ into $D$ which uses only the edges of $R$.
As remarked before, we can assume that $\phi$ is a uniformly random injective mapping from $V(T')$ to $V(D)$.

By Lemma \ref{lemma:findingstars}, we can find $V \subseteq V(T')$ of size $\gamma n$ such that $\mathcal{S} = \{ (v, N_{T'}^+(v), N_{T'}^-(v)) | v \in V\}$ is a set of vertex-disjoint stars in $T'$.
Using Lemma \ref{lemma:enoughstars}, we moreover get that w.h.p. for all $u, w \in V(D)$ and $* \in \{ +, - \}$ there are at least $c_1 n$ many $(u, *, w)$-absorbing stars in $\mathcal{S}$ under $\phi'$.
By Lemma \ref{lemma:absorbing}, since $\epsilon n < 2 c_1 n$, we can therefore extend $\phi'$ to an embedding $\phi: V(T) \to V(D)$ of $T$ into $D$, which shows that w.h.p. $T \subseteq D$. 
\end{proof}

\bibliographystyle{plain}
\bibliography{refs}

\end{document}